\documentclass[12pt]{article}
\usepackage{amssymb,amsmath}
\usepackage{amsthm}
\usepackage{amsfonts}
\usepackage{latexsym}
\begin{document}
\baselineskip18pt
\newcommand{\h}{\mathscr{H}}
\newcommand{\m}{\text{M}}
\newcommand{\uu}{\textbf{u}}
\newcommand{\OO}{\Omega^*}
\newcommand{\xx}{\textbf{X}}
\newcommand{\R}{\text{Ra}}
\newcommand{\dd}{\mathcal{D}}
\newtheorem{theorem}{\sf Theorem}[section]
\newtheorem{lemma}[theorem]{\sf Lemma}
\newtheorem{proposition}[theorem]{\sf Proposition}
\newtheorem{corollary}[theorem]{\sf Corollary}
\newenvironment{definition}[1][\sf Definition:]{\begin{trivlist}
\item[\hskip \labelsep {\bfseries #1}]}{\end{trivlist}}
\newenvironment{example}[1][\sf Example:]{\begin{trivlist}
\item[\hskip \labelsep {\bfseries #1}]}{\end{trivlist}}
\newenvironment{examples}[1][\sf Examples:]{\begin{trivlist}
\item[\hskip \labelsep {\bfseries #1}]}{\end{trivlist}}
\newenvironment{remark}[1][\sf Remark:]{\begin{trivlist}
\item[\hskip \labelsep {\bfseries #1}]}{\end{trivlist}}
\numberwithin{equation}{section}
\parindent=0cm
\title{Sums of products involving power sums of $\varphi(n)$ integers}
\date{}
\author{Jitender Singh$^1$}
\maketitle
\begin{abstract}
A sequence of rational numbers as a generalization of the sequence of Bernoulli numbers is introduced. Sums of products involving the terms of this generalized sequence are then obtained using an application of the Fa\`a di Bruno's formula. These sums of products are analogous to the higher order Bernoulli numbers and are used to  develop the closed form expressions for the sums of products 
involving the power sums $\displaystyle \Psi_k(x,n):=\sum_{d|n}\mu(d)d^k S_k\left(\frac{x}{d}\right), n\in\mathbb{Z}^+$ which are defined via the M\"obius function $\mu$ and the usual power sum $S_k(x)$ of a
real or complex variable $x.$  The  power sum $S_k(x)$ is expressible in terms of the well known Bernoulli polynomials by $\displaystyle S_k(x):=\frac{B_{k+1}(x+1)-B_{k+1}(0)}{k+1}.$
\end{abstract}
\footnotetext[1]{Department of Mathematics, Guru Nanak Dev
University, Amritsar-143005, INDIA\\
{\tt sonumaths@gmail.com;~jitender.math@gndu.ac.in}}
\section{Introduction}
Singh \cite{JS9} introduced the power sum $\Psi_k(x,n)$ of real or complex variable $x$ and positive integer $n$ defined by
the generating function
\begin{equation}\label{eq5}
\sum_{d|n}\mu(d)\frac{e^{\left(1+\frac{x}{d}\right)td}-e^{td}}{e^{td}-1}=\sum_{k=0}^{\infty}\Psi_k(x,n)\frac{t^k}{k!}
\end{equation}
from which he derived the following closed form formula for these power sums:
\begin{equation}\label{eq5}
\Psi_k(x,n)=\frac{1}{k+1}\sum_{m=0}^{[\frac{k}{2}]}\left(
                          \begin{array}{c}
                            k+1 \\
                            2m\\
                          \end{array}
                        \right)B_{2m}x^{k+1-2m}\prod_{p|n}(1-p^{2m-1}),~k=0,1,...
\end{equation}
where $B_m$ are the Bernoulli numbers and $p$ runs over all prime  divisors of $n.$
In particular $\Psi_k(n,n)$ gives the sum of $k-$th power of those positive integers which are less then $n$ and relatively prime to $n.$ We will call $\Psi_k(x,n)$ as \emph{M\"obius-Bernoulli power sums}. Present work is aimed at describing sums of products of the power sums $\Psi_k(x,n)$ via introducing yet another sequence of rational numbers which we shall call as the sequence of \emph{M\"obius-Bernoulli numbers}. The rational sequence $\{B_k\}$ that appears in Eq.\eqref{eq5} is defined via the generating function
\begin{equation*}\label{eq3}
    \frac{t}{e^t-1}=\sum_{k=0}^{\infty}B_k \frac{t^k}{k!},~|t|<2\pi
\end{equation*}
and was known to Faulhaber and Bernoulli. 
Many explicit formulas for the Bernoulli numbers are also well known in literature. One such formula is the following \cite{h}:
\begin{equation*}\label{eq4}
B_k=\sum_{m=1}^{k}\frac{1}{m+1}\sum_{n=1}^{m}(-1)^n \left(
                          \begin{array}{c}
                            m \\
                            n\\
                          \end{array}
                        \right) n^k,~k=0,1,...
\end{equation*}
\section{M\"obius-Bernoulli numbers}
\begin{definition}We define M\"obius-Bernoulli numbers $M_k(n),~k=0,1,...$ via the generating function
    \begin{equation}\label{eq11}
      \sum_{d|n} \frac{t \mu(d)}{e^{td}-1}=\sum_{k=0}^{\infty}M_k(n)\frac{t^k}{k!},~|t|<\frac{2\pi}{n}~\forall~n=1,2,...
    \end{equation}
\end{definition}
We immediately notice from the Eq. \eqref{eq11} that the M\"obius-Bernoulli numbers are given by
\begin{equation}\label{eq12}
      M_k(1)=B_k;~M_k(n)=B_k\prod_{p|n}(1-p^{k-1}) ~\mbox{for all}~ n\geq 2,~k=0,1,....
    \end{equation}
Singh\cite{JS9} has obtained the following identity relating the function $\Psi_k(x,n)$ to the M\"obius-Bernoulli numbers. 
$\displaystyle \frac{d}{dx}\Psi_k(x,n)=k\Psi_{k-1}(x,n)+(-1)^k M_k(n),$ from which we observe that $\Psi_k(x,1)=S_k(x),$ $\Psi_0(x,n)=\varphi(n),~\Psi_k(0,n)=0~\forall~k=0,1,...$ where $\varphi(n)$ is the Euler's totient.
\begin{definition} Let $n$ be a positive integer and $k$ a nonnegative integer. Define higher order M\"obius-Bernoulli numbers by
$$\displaystyle M_k^N(n):=\sum_{\sum_{i=1}^{N}k_i=k}\left(\begin{array}{c}
                                                   k \\
                                                   k_1,...,k_N
                                                 \end{array}\right) M_{k_1}(n)\cdots M_{k_N}(n),
$$ which are described by the generating function
    \begin{equation}\label{eq13}
    H(t)=  \left(\sum_{d|n} \frac{t \mu(d)}{e^{td}-1}\right)^N=\sum_{k=0}^{\infty}M_k^N(n)\frac{t^k}{k!},~|t|<\frac{2\pi}{n}~\forall~n,N=1,2,...
    \end{equation}
\end{definition}
Note that $M_k^1(n)=M_k(n)~\forall~k=0,1,...$ Some of the first few higher order M\"obius-Bernoulli numbers are given by the following:
\begin{table}[h!]
\begin{tabular}{ll}
  $M_0^N(n)$ &$=(\varphi(n))^{N}$ \\
  $M_2^N(n)$&$=N(\varphi(n))^{N-1}M_2 (n)$ \\
  $M_4^N(n)$&$=N(\varphi(n))^{N-2}\left\{3(N-1) (M_2 (n))^2+\varphi(n)M_4(n)\right\}$ \\
  $M_6^N(n)$&$=N(\varphi(n))^{N-3}\left\{15(N-1)(N-2)(M_2 (n))^2+\right.$\\ & $\left. 15(N-1)\varphi(n)M_2(n)M_4(n)+\varphi(n)^2 M_6(n)\right\}$\\
 etc.
\end{tabular}
\end{table}\\
Note that $\displaystyle M_k^N(n)=\lim_{t\rightarrow 0} \frac{d^kH(t)}{dt^k}.$ In this regard, a formula for the higher order M\"obius-Bernoulli numbers can be obtained from the following version of the well known Fa\`a di Bruno's formula \cite{johnson}.
\begin{lemma}\label{L1}
    Let $N$ be a positive integer and $f:\mathbb{R}\rightarrow \mathbb{R}$ be a function of class $C_k,~k\geq 1.$ {\fontsize{10}{12} Then
\begin{equation}\label{eq91}
    D^k(f(x))^N=N!\sum_{j=1}^{k} \frac{(f(x))^{N-j}}{(N-j)!}\sum_{\sum_{i=1}^{j}k_i=k}\left(\begin{array}{c}
                                                      k \\
                                                      k_1,...,k_j
                                                    \end{array}\right) \frac{D^{k_1}f(x)\cdots D^{k_j}f(x)}{\lambda(k_1)!\cdots \lambda(k_j)!}
\end{equation}}
where $D^k=\frac{d^k}{dx^k}, k=1,2,...;$ $\lambda(k_i)$ is the multiplicity of occurrence of $k_i$ in the partition $\{k_1,...,k_j\}$ of $n$ of length $j$ and $\lambda(k_i)!$ contributes only once in the above product.
\end{lemma}
\begin{proof} We use induction on $k$ in proving the result. For $k=1,$ we see that $j=1$  in the RHS of Eq.\eqref{eq91}  and  it reduces to $N (f(x))^{N-1}D^1 (f(x))=D^1(f(x)^N).$ This proves that the result is true for $k=1.$ Let us assume that the formula \eqref{eq91} holds for all positive integers $\leq k.$ Now assume that $f$ is of class $C_{k+1}$ and consider
{\fontsize{10}{0} \begin{equation}\label{eq92}
    \begin{split}
    D^{k+1}((f(x))^N)
               =&N!\sum_{j=1}^{k-1} \frac{(f(x))^{N-(j+1)}}{(N-(j+1))!}\sum_{\sum_{i=1}^{j}k_i=k}\left(\begin{array}{c}
                                                      k \\
                                                      k_1,...,k_j
                                                    \end{array}\right) \frac{D^1f(x)D^{k_1}f(x)\cdots D^{k_j}f(x)}{\lambda(k_1)!\cdots \lambda(k_j)!}\\                                     +&N!\frac{(f(x))^{N-k-1}}{(N-k-1)!}(Df(x))^{k+1}
                                                    \\+&N!\sum_{j=1}^{k} \frac{(f(x))^{N-j}}{(N-j)!}\sum_{\sum_{i=1}^{j}k_i=k}\left(\begin{array}{c}
                                                      k \\
                                                      k_1,...,k_j
                                                    \end{array}\right) \frac{D(D^{k_1}f(x)\cdots D^{k_j}f(x))}{\lambda(k_1)!\cdots \lambda(k_j)!}
    \end{split}
\end{equation}}
At this point observe that any partition $\pi'$ of $k+1$ can be  obtained from a partition $\pi$ of $k$ by adjoining $1$ and let us denote the set of all such partitions of $k+1$ by $S.$ Denote by $T$ the set of remaining all partitions of $k+1$ where each  $\pi'$ is obtained simply by adding $1$ to exactly one member of $\pi.$
In each of these cases one has $k+1$ choices of doing so for a fixed $\pi.$ In the former case for each $\pi'\in S,$ $|\pi'|=|\pi|+1$ which happens in the first summation above in \eqref{eq92} which reduces to the following
{\fontsize{10}{0}\begin{equation}\label{eq93}
  (k+1)N!\sum_{j=1}^{k} \frac{(f(x))^{N-j}}{(N-j)!}\sum_{\{k_1',\cdots,k_{j+1}'\}\in S}\left(\begin{array}{c}
                                                      k \\
                                                      k_1',...,k_{j+1}'
                                                    \end{array}\right) \frac{D^{k_1'}f(x)\cdots D^{k_{j+1}'}f(x)}{\lambda(k_1')!\cdots \lambda(k_j')!}.
\end{equation}}
In the latter case for each $\pi'\in T,$ $|\pi'|=|\pi|$ and the terms after first summation in \eqref{eq92} reduce to {\fontsize{10}{0}
\begin{equation}\label{eq94}
  (k+1)N!\sum_{j=1}^{k+1} \frac{(f(x))^{N-j}}{(N-j)!}\sum_{\{k_1',\cdots,k_{j+1}'\}\in T}\left(\begin{array}{c}
                                                      k \\
                                                      k_1',...,k_{j+1}'
                                                    \end{array}\right) \frac{D^{k_1'}f(x)\cdots D^{k_{j+1}'}f(x)}{\lambda(k_1')!\cdots \lambda(k_j')!}
\end{equation}}
where the term $\displaystyle N!\frac{(f(x))^{N-k-1}}{(N-k-1)!}(Df(x))^{k+1}$ corresponds to $j=k+1.$ The result follows by substituting \eqref{eq93} and \eqref{eq94} in \eqref{eq92}.  This completes the final step of induction.
\end{proof}
\begin{theorem}\label{th12} For each positive integer $N$ $\&$ $n,$ the higher order M\"obius-Bernoulli numbers are given by
    {\fontsize{11}{0}\begin{equation}\label{eq95}
        M_{k}^N(n)=N!\sum_{j=1}^{k} \frac{(\varphi(n))^{N-j}}{(N-j)!}\sum_{\sum_{i=1}^{j}k_i=k}\left(\begin{array}{c}
                                                      k \\
                                                      k_1,...,k_j
                                                    \end{array}\right) \frac{M_{k_1}(n)\cdots M_{k_j}(n)}{\lambda(k_1)!\cdots \lambda(k_j)!}.
    \end{equation}}
\end{theorem}
\begin{proof}
First note from definition that $\displaystyle M_k^N(n)=\lim_{t\rightarrow 0}D^k (H(0)).$ The result follows at once by applying Lemma \ref{L1} to the function $\displaystyle f(t)=\sum_{d|n}\frac{t\mu(d)}{e^{td}-1},~\displaystyle 0<t<\frac{2\pi}{n}$ and then taking limit $t\rightarrow 0$ throughout and using $\displaystyle \varphi(n)=\sum_{d|n}\frac{\mu(d)}{d}$ therein.
\end{proof}
\begin{proposition}\label{p1}
$M_{2k-1}^N(n)=0$ for all positive integers $k$ $\&$ $N$ and $n>1.$
\end{proposition}
\begin{proof}
Observe from Eq.\eqref{eq13}  that for a positive integer $N,$ the following holds:
$$H(-t)=\left(t\sum_{d|n}\mu(d)+ \sum_{d|n}\frac{t \mu(d)}{e^{td}-1}\right)^N=\left(t\delta_{1n}+ \sum_{d|n}\frac{t \mu(d)}{e^{td}-1}\right)^N=H(t)$$ for all $n>1$ where the arithmetic function  $\sum_{d|n}\mu(d)=\delta_{1n}$ is the Kronecker delta. We have proved that $H$ is an even function of $t$ for $n>1.$ Thus  the coefficient of $t^{2k-1}$ in the RHS of Eq.\eqref{eq13} (which is precisely $M_{2k-1}^N(n)$) vanishes for each $k=1,2,...$
\end{proof}
\begin{remark} If we extend the definition of higher M\"obius-Bernoulli numbers to complex $N\neq 0,$
the formula \eqref{eq95} for $M_k^N(n)$ is still valid just on replacing $\frac{N!}{(N-j)!}$ by ${N(N-1)\cdots (N-j+1)}$ in it. In this regard we note that $M_{2k-1}^N(n>1)=0$ holds for all $k=1,2,...$ and $N\in\mathbb{C}.$
\end{remark}
In view of the theorem \ref{th12} and the proposition \ref{p1}, we have for all positive integers $n>1$ and $k=1,2,...,$
 {\fontsize{10}{0}\begin{equation}\label{eq95}  
        M_{2k}^N(n)=N!\sum_{j=1}^{k} \frac{(\varphi(n))^{N-{2j}}}{(N-2j)!}\sum_{\sum_{i=1}^{j}k_i=k}\left(\begin{array}{c}
                                                      2k \\
                                                      2k_1,...,2k_j
                                                    \end{array}\right) \frac{M_{2k_1}(n)\cdots M_{2k_j}(n)}{\lambda(2k_1)!\cdots \lambda(2k_j)!}.
    \end{equation}}
    As an example, let $k=4.$ There are five partitions of $4$ which are given by
    ${{4_1}, {3_11_1}, {2_2}, {2_11_1}, 1_4}$ and therefore from Eq.\eqref{eq95} we obtain
 \begin{equation*}\label{eq96}  
 \begin{split}
        M_{8}^N(n)&=N!\{\frac{\varphi(n)^{N-2}}{(N-2)!}M_{8}(n)+
        \frac{\varphi(n)^{N-4}}{(N-4)!}(\frac{8!}{6! 2!}M_6(n) M_2(n)+\frac{8!}{4!4!}\frac{(M_{4}(n))^2}{2!})\\&+ \frac{\varphi(n)^{N-6}}{(N-6)!}(\frac{8!}{4!2!2!}M_4(n)\frac{(M_2(n))^2}{2!})+\frac{\varphi(n)^{N-8}}{(N-8)!}\frac{8!}{2!2!2!2!}
        \frac{(M_2(n))^4}{4!}\}\\
        &=  N!\{\frac{\varphi(n)^{N-2}}{(N-2)!}M_{8}(n)+
        \frac{\varphi(n)^{N-4}}{(N-4)!}(28 M_6(n) M_2(n)+35{(M_{4}(n))^2})\\&+ \frac{\varphi(n)^{N-6}}{(N-6)!}(210M_4(n){(M_2(n))^2})+\frac{\varphi(n)^{N-8}}{(N-8)!}105
        {(M_2(n))^4}\}\\
                                                    \end{split}
                                                    \end{equation*}
\begin{remark}
The formula \eqref{eq95}  is not suitable for explicit evaluation of $M_k^N$ for large $k.$
Because number of partitions of $k$ increases at a faster rate than $k.$ For example number of partitions of $10$ is $42,$ which is the number of terms in the expression for $M_{20}^N.$ In this regards, it will be good to see a formula for the higher order M\"obius Bernoulli numbers which can describe them better than the one we have given above!
\end{remark}
\begin{remark} If $n=p^s$ for some positive integer $s$ and prime $p$, then the simplest possible formula \eqref{eq95} for the higher order M\"obius-Bernoulli numbers can be found as follows:
{\fontsize{10}{0}$$H(t)|_{n=p^s}=\left(\frac{t}{e^t-1}-\frac{t}{e^{tp}-1}\right)^N=\sum_{m=0}^{N}\left(\begin{array}{c}
                                                                         N \\
                                                                         m
                                                                       \end{array}\right)\left(\frac{t}{e^t-1}\right)^{m}
                                                                       \left(\frac{tp}{e^{tp}-1}\right)^{N-m} (-p)^{m-N}
$$}
Therefore
{\fontsize{10}{0}\begin{equation}\label{eq97}
M_{k}^{N}(p^s)=\lim_{t\rightarrow 0}\frac{d^kH(t)}{dt^k}=\sum_{m=0}^{N}\left(\begin{array}{c}
                                                                         N \\
                                                                         m
                                                                       \end{array}\right)(-p)^{m-N}\sum_{j=0}^{k}\left(\begin{array}{c}
                                                                         k \\
                                                                         j
                                                                       \end{array}\right)B_j^m B_{k-j}^{N-m} p^{k-j}
\end{equation}}
where we have utilized the Leibniz product rule for higher order derivatives and $B_j^m$ is the higher order Bernoulli number given by (see for more details Srivastava and Todorov \cite{p0})
\begin{equation*}\label{eq00}
   B_j^m=\sum_{\ell=0}^{j}\left(\begin{array}{c}
      j+m \\
      j-\ell
    \end{array}\right)\left(\begin{array}{c}
      j+m-1 \\
      \ell
    \end{array}\right)\frac{j!}{(j+\ell)!}\sum_{h=0}^{\ell}(-1)^{h}\left(\begin{array}{c}
      \ell \\
      h
    \end{array}\right)h^{j+\ell}.
\end{equation*}
Similarly if we take $n=p_1^{s_1}p_2^{s_2}$ for some positive integers $s_1,s_2$ and distinct primes $p_1,p_2$ then
{\fontsize{11}{0}\begin{equation*}\label{eq98}
\begin{split}
    H(t)|_{n=p_1^{s_1}p_2^{s_2}}&=\left(\frac{t}{e^t-1}-\frac{t}{e^{tp_1}-1}-\frac{t}{e^{tp_2}-1}+\frac{t}{e^{tp_1p_2}-1}\right)^N\\&=
\sum_{\sum_{i=0}^3m_i=N}\left(\begin{array}{c}
                                                                         N \\
                                                                         m_0,\cdots,m_3
                                                                       \end{array}\right)
                                                                       \left(\frac{t}{e^t-1}\right)^{m_0}
                                                                       \left(\frac{tp_1}{e^{tp_1}-1}\right)^{m_1}
                                                                       \left(\frac{tp_2}{e^{tp_2}-1}\right)^{m_2}\\&\times
                                                                       \left(\frac{tp_1p_2}{e^{tp_1p_2}-1}\right)^{m_3}
                                                                       (-p_1)^{-m_1}(-p_2)^{-m_2} (p_1p_2)^{-m_3}
\end{split}
\end{equation*}}
which gives
{\fontsize{11}{0}\begin{equation*}\label{eq98}
\begin{split}
    M_k^{N}(p_1^{s_1}p_2^{s_2})&=
\sum_{\sum_{i=0}^3m_i=N}\left(\begin{array}{c}
                                                                         N \\
                                                                         m_0,\cdots,m_3
                                                                       \end{array}\right){\lim_{t\rightarrow 0}}D^k\left\{
                                                                       \left(\frac{t}{e^t-1}\right)^{m_0}
                                                                       \left(\frac{tp_1}{e^{tp_1}-1}\right)^{m_1}\right.\\&\times \left.
                                                                       \left(\frac{tp_2}{e^{tp_2}-1}\right)^{m_2}
                                                                       \left(\frac{tp_1p_2}{e^{tp_1p_2}-1}\right)^{m_3}\right\}
                                                                       (-p_1)^{-m_1}(-p_2)^{-m_2} (p_1p_2)^{-m_3}\\&=
\sum_{\sum_{i=0}^3m_i=N}\left(\begin{array}{c}
                                                                         N \\
                                                                         m_0,\cdots,m_3
                                                                       \end{array}\right)\sum_{\sum_{j=0}^3k_i=k}\left(\begin{array}{c}
                                                                         k \\
                                                                         k_0,\cdots,k_3
                                                                       \end{array}\right)\\&\times B_{k_0}^{m_0}B_{k_1}^{m_1}B_{k_2}^{m_2}B_{k_3}^{m_3}
                                                                       p_1^{k_1-m_1}p_2^{k_2-m_2}(p_1p_2)^{k_3-m_3}(-1)^{m_1+m_2}
\end{split}
\end{equation*}}
These formulas involve products of higher order Bernoulli numbers. So in general, the formulas for $M_k^N(n)$ involve sums containing product of several higher order Bernoulli numbers and such a formula in the above sense would be complicated and will take the following form:
{\fontsize{11}{0}
\begin{equation*}\label{eq99}
\begin{split}
    M_k^{N}\left(\prod_{\alpha=1}^{\beta}p_\alpha^{e_\alpha}\right)&=
\frac{N!k!}{(N+k)!}\sum_{\sum_{i=0}^{2^\beta-1}m_i=N}
\sum_{\sum_{j=0}^{2^\beta-1}k_i=k}\left(\begin{array}{c}
                                                                         N+k \\
                                                                        m_0,...,m_{2^\beta-1}, k_0,...,k_{2^\beta-1}
                                                                       \end{array}\right)\\&\times B_{k_0}^{m_0}B_{k_1}^{m_1}
                                                                       \cdots B_{k_{2^\beta-1}}^{m_{2^\beta-1}} \chi(k_1,...,k_{2^\beta-1},m_1,...,m_{2^\beta-1})
\end{split}
\end{equation*}}
where
{\fontsize{10}{0}\begin{equation*}\label{eq99}
\begin{split}
    \displaystyle \chi(k_1,...,k_{2^\beta-1},m_1,...,m_{2^\beta-1})&=\prod_{s_0=1}^{\beta}(-p_{s_0})^{k_{s_0}-m_{s_0}}\cdots
\prod_{s_{\beta-2}<s_{\beta-1}=\beta}^{\beta
}((-1)^\beta p_{s_0}p_{s_1}\cdots p_{s_{\beta-1}})^{k_{2^{\beta}-1}-m_{2^\beta-1}}.
\end{split}
\end{equation*}}
\end{remark}
\section{Sums of products of M\"obius-Bernoulli power sums}
Having developed the expressions for the M\"obius Bernoulli numbers in the previous section, we will now use them in expressing the sums of products of the M\"obius-Bernoulli power sums $\Psi_k(x,n).$
\begin{definition}
We define sums of products of the M\"obius Bernoulli power sums as $\Psi_k^N(x,n):=\displaystyle \sum_{\sum_{i=1}^{N}k_i=k} \left(
                          \begin{array}{c}
                            k \\
                            k_1~\cdots~ k_N  \\
                          \end{array}
                        \right)\Psi_{k_1}(x,n)\cdots \Psi_{k_N}(x,n)$ for nonnegative integers $k$ and $N$ which are described by the generating function
                        \begin{equation}
                            \left(\sum_{d|n} \mu(d)\frac{e^{(1+\frac{x}{d})td} -e^{t d}}{e^{td}-1}\right)^N=\sum_{k=0}^{\infty}\Psi_k^N(x,n)\frac{t^k}{k!}.
                        \end{equation}

\end{definition}
The next result evaluates the sums of products $\Psi_k^N(x,n).$
\begin{theorem}\label{th2} For a positive integer $N$ and nonnegative integer $k,$
    $$\Psi_k^N(x,n)=\begin{array}{cl}
\displaystyle \frac{k! N!}{(k+N)!}\sum_{j=0}^{k}\left(\begin{array}{c}
                                                   k+N \\
                                                   j
                                                 \end{array}\right)
                                                 M_{j}^{N}(n)  S(k+N-j,N) x^{k+N-j}, & ~\end{array}$$ for all ~$n=2,3,...$
                                                 where $S(\ell,m)$ are the Stirling numbers of second kind.
\end{theorem}
Note that from theorem \ref{th2} we recover for $N=1$
$$\Psi_k^1(x,n)=\frac{1}{k+1}\sum_{j=0}^{k}\left(\begin{array}{c}
                                                   k+1 \\
                                                   j
                                                 \end{array}\right)
                                               x^{k+1-j}  B_j\prod_{p|n}(1-p^{j-1}) ,~\forall~n=2,3,...$$
where we have used $M_j^1(n)=B_j\prod_{p|n}(1-p^{j-1}),$ and  $S(k+1-j,1)=1$ for all $j=0,1,,...k.$
\begin{proof} Observe from the generating function for $\Psi_k^N(x)$ that
{\fontsize{10}{0}\begin{equation}\label{eq14}
 \begin{split}
      \left(\sum_{d|n} \mu(d)\frac{e^{(1+\frac{x}{d})td} -e^{t d}}{e^{td}-1}\right)^N&=\left(\sum_{d|n} \mu(d)\frac{t e^{td}}{e^{td}-1}\right)^N \left(\frac{e^{xt}-1}{t}\right)^N\\
      &=\left(t\sum_{d|n} \mu(d)+\sum_{d|n}\frac{t\mu(d)}{e^{td}-1}\right)^N \left(\frac{e^{xt}-1}{t}\right)^N\\
      &=\left(t\delta_{1n}+\sum_{d|n}\frac{t\mu(d)}{e^{td}-1}\right)^N \left(\frac{e^{xt}-1}{t}\right)^N\\
      &=\sum_{m=0}^{N}\left(\begin{array}{c}
                                                   N \\
                                                   m
                                                 \end{array}\right)t^{N-m}\delta_{1n}^{N-m}\left(
                                                 \sum_{d|n}\frac{t\mu(d)}{e^{td}-1}\right)^m \left(\frac{e^{xt}-1}{t}\right)^N
                                                 \end{split}
                                                 \end{equation}
                                                 which on further simplifications give
                                                 \begin{equation}
      \sum_{k=0}^{\infty}\Psi_k^N(x,n)\frac{t^k}{k!} =\sum_{k=0}^\infty \sum_{m=0}^{N}\left(\begin{array}{c}
                                                   N \\
                                                   m
                                                 \end{array}\right)\sum_{j=0}^{k}\left(\begin{array}{c}
                                                   k \\
                                                   j
                                                 \end{array}\right) \delta_{1n}
                                                 M_{k-j}^{m}(n) \frac{j! N!}{(j+N)!} S(j+N,N) x^{j+N}\frac{t^k}{k!}
    \end{equation}
  }where we have used the identity $\displaystyle \left(\frac{e^{xt}-1}{t}\right)^N=\sum_{j=0}^\infty \frac{k! N!}{(j+N)!} S(j+N,N) x^{j+N}\frac{t^{j}}{j!}.$ On comparing like powers of $t$ in Eq.\eqref{eq14} gives
  $$\displaystyle \Psi_k^N(x,n)=\sum_{m=0}^{N}\left(\begin{array}{c}
                                                   N \\
                                                   m
                                                 \end{array}\right)\sum_{j=0}^{k}\left(\begin{array}{c}
                                                   k \\
                                                   j
                                                 \end{array}\right) (\delta_{1n})^{N-m}
                                                 M_{k-j}^{m}(n) \frac{j! N!}{(j+N)!} S(j+N,N) x^{j+N}$$
for all $k=0,1,...; ~n, N=1,2,...$ where we define $\delta_{1n}^0:=1$ for all $n=1,2,...$ The result follows now.
\end{proof}
\subsection*{Acknowledgements}
Many suggestions regarding presentation of the paper by Professor L\'aszl\'o T\'oth are gratefully acknowledged.
\bibliographystyle{plain}

\end{document}